%% file: connectionPath.tex
\documentclass[12pt,twoside,a4paper,russian,english]{article}

\usepackage{fullpage}

\usepackage[utf8]{inputenc}
\usepackage[T2A,T1]{fontenc}
\usepackage{xspace}
\usepackage{calrsfs}
\usepackage{bbm}
\usepackage{color}
\usepackage{floatflt}
\definecolor{myblue}{rgb}{0.09,0.32,0.44} 
\usepackage[pdfborder={0 0 0},pdfborderstyle={},colorlinks=true,linkcolor=myblue,citecolor=myblue,urlcolor=blue]{hyperref}
\usepackage{breakurl}
\ifpdf
\def\afs#1#2{\href{#1}{\nolinkurl{#2}}}
\else
\def\afs#1#2{\burlalt{#1}{#2}}
\fi
\usepackage[english]{babel}

\usepackage{amsthm}

\theoremstyle{plain}
\newtheorem{theorem}{Theorem}

\newtheorem{proposition}[theorem]{Proposition}

\newtheorem{lemma}[theorem]{Lemma}

\newtheorem{conjecture}{Conjecture}

\theoremstyle{remark} 
\newtheorem*{remark}{Remark}

\theoremstyle{definition} 
\newtheorem*{definition}{Definition}

\usepackage{amsmath}
\usepackage{amssymb}
\usepackage{mathtools}
\mathtoolsset{showonlyrefs}
\newcommand{\eps}{\varepsilon}
\newcommand{\ep}{\varepsilon}
\newcommand{\bbP}{\mathbb P}

\newcommand{\Z}{\mathbb{Z}}

\newcommand{\N}{\mathbb{N}}
\newcommand{\PP}{\mathbb{P}}

\newcommand{\Pp}[2][p]{\PP_{#1} \left [ #2 \right ]}
\newcommand{\lr}[1][]{\xleftrightarrow{#1}}
\newcommand{\nlr}[1][]{\overset{#1}{\nleftrightarrow}}

\DeclareMathOperator{\slab}{Slab}
\DeclareMathOperator{\oslab}{OSlab}

\usepackage{tabularx}
\usepackage{graphicx}

\title{Upper bounds on the percolation correlation length} \author{Hugo Duminil-Copin\footnotemark[1]\footnote{Universit\'e de Gen\`eve} \footnotemark[2]\footnote{Institut des Hautes \'Etudes Scientifiques}  , Gady Kozma\footnotemark[3]\footnote{Weizmann Institute} , Vincent Tassion\footnotemark[4]\footnote{ETH Zurich}}
\date{}

\begin{document}
\maketitle
\bigbreak
\bigbreak
\noindent
\begin{center}
  {\em In memory of our dear friend Vladas Sidoravicius 1963--2019}
\end{center}
\bigbreak
\bigbreak
\begin{abstract}
We study the size of the near-critical window for Bernoulli percolation on $\Z^d$. More precisely, we use a quantitative Grimmett-Marstrand theorem to prove that the correlation length, both below and above criticality, is bounded from above by $\exp(C/|p-p_c|^2)$. Improving on this bound would be a further step towards the conjecture that there is no infinite cluster at criticality on $\Z^d$ for every $d\ge2$.
\end{abstract}

\section{Introduction}

\subsection{Critical percolation}\label{sec:crit}

The main open question in percolation theory is to understand the behaviour at criticality, i.e.~when $p$ is equal to $p_c$, and in particular to prove that there does not exist an infinite cluster at $p_c$ (precise definitions will be given below, in \S\ref{sec:prelim}).
\begin{conjecture}\label{conj:1}
For every $d\ge2$, $\mathbb P_{p_c}[0\leftrightarrow\infty]=0$.
\end{conjecture}

This conjecture has been solved for $d=2$ \cite{Har60,Kes80} and for $d\ge11$ \cite{FitHof15} based on ideas pioneered in  \cite{BrySpe85,HarSla90}. An important result related to the techniques of our paper is the fact that there is no percolation on a half space \cite{BarGriNew91}. Further, the result is also known for graphs of the form $\mathbb Z^2\times G$ with $G$ finite; see \cite{DamNewSid15,DumSidTas14}. On transitive graphs with rapid growth, additional tools are available, and the following cases are known: non-amenable graphs \cite{BenLyoPer99}, graphs with exponential growth \cite{Hut16}, and recently some graphs with stretched-exponential growth \cite{HutHer18}.

A natural scheme to attack the conjecture on $\Z^d$ is to find a $\delta>0$ and a sequence of events $\mathcal E_n$ depending on edges in the box $\Lambda_n\coloneqq\{-n,\dotsc,n\}^d$ only, such that {\em for any $p$},
\begin{equation}\tag{$\star$}\label{eq:star} \exists n>0\text{ s.t.~}\mathbb P_p[\mathcal E_n]>1-\delta\quad\Longleftrightarrow\quad\mathbb P_p[0\leftrightarrow\infty]>0.\end{equation}
If such a sequence exists, the set of $p$ such that $\mathbb P_p[0\leftrightarrow\infty]>0$ is an open set since it is the union of the open sets (indexed by $n$) $\{p:\mathbb P_p[\mathcal E_n]>1-\delta\}$ (this set is open since $p\mapsto \mathbb P_p[\mathcal E_n]$ is continuous). 

Of course, this strategy is tempting, but the main difficulty is that the $\Longrightarrow$ and the $\Longleftarrow$ implications involved in \eqref{eq:star} are difficult to prove simultaneously. One may for instance easily check the $\Longrightarrow$ implication by asking a lot on $\mathcal E_n$, but then the $\Longleftarrow$ one becomes difficult, and vice-versa. To illustrate this trade-off phenomenon, let us give a few examples of possible sequences $(\mathcal E_n)$, going from the strongest criterion (meaning the one for which the $\Longrightarrow$ implication is the easiest to prove) to the weakest one (meaning the one for which $\Longrightarrow$ is the hardest).
\medbreak\noindent
{\em Example 1.} Let $\mathcal E_n$ be the event that $\Lambda_{n/10}$ is connected to $\partial\Lambda_n\coloneqq\Lambda_n\setminus\Lambda_{n-1}$ and that the second largest cluster in $\Lambda_n$ has radius smaller than $n/10$. In this case, a coarse-graining argument similar to \cite{AntPis96} implies the $\Longrightarrow$ implication easily. Proving $\Longleftarrow$ is still open in particular because of the difficulty to exclude the existence of many large clusters avoiding each other.
\medbreak\noindent
{\em Example 2.} Let $\mathcal E_n$ be the intersection of the events that 
$(\pm n,0)+\Lambda_{n/2}$ are connected in $\Lambda_{2n}$ and that there exists at most one cluster in $\Lambda_{2n}$ going from $(\pm n,0)+\Lambda_{n/2}$ to $(\pm n,0)+\partial\Lambda_{n}$.
A coarse-graining argument may be used to prove $\Longrightarrow$ but $\Longleftarrow$ remains open due to the same reason as the previous condition. 
\medbreak
In general, uniqueness of clusters going from one area to another one is a key difficulty in these problems. This might be related to the fact that in high dimensions $\Lambda_n$ indeed hosts many disjoint clusters in $p_c$, see \cite{A97}. In order to circumvent this difficulty, one can make different choices for $\mathcal E_n$.  \medbreak\noindent
{\em Example 3.} Let $\mathcal E_n$ be the same event as in the second example, but with $(\pm n,0)+\Lambda_{n/2}$ replaced by $(\pm n,0)+\Lambda_{u_n}$, with $u_n$ much smaller than $n/2$. In this case, the implication $\Longrightarrow$ is as before, and does not depend on $u_n$. As for the implication $\Longleftarrow$, as $u_n$ becomes smaller the connectivity part becomes harder and the uniqueness part becomes easier. 
\medbreak
 Recently, a paper of Cerf \cite{MR3395466}, based on \cite{AizKesNew87}, provided a beautiful insight on how big $u_n$ must be taken to have that with large probability, $\Lambda_n\coloneqq\{-n,\dotsc,n\}^d$, the box of size $n$, contains at most one cluster going from $\Lambda_{u_n}$ to $\partial\Lambda_n$. We will come back to this later in the introduction, but let us mention the result right now.
 
 \newcommand{\piv}{A_2} 
 Given $1\le m \le n$, consider the set of clusters in the configuration restricted to the box $\Lambda_n$, and define $A_2(m,n)$ to be  the event that there exists at least two disjoint such clusters intersecting both $\Lambda_m$ and $\partial\Lambda_n$. 
\begin{proposition}[Cerf]\label{prop:uniqueness} Let $d\ge2$. There exists
  $\alpha=\alpha(d)\in (0,1)$  such that for any $p\in[0,1]$ and $n$ large enough,
  \begin{equation}
 \Pp{A_2(n^{\alpha},n)}\le\frac1{n^\alpha}.\label{eq:1}
\end{equation}
\end{proposition}
We fill some details on this proposition in \S\ref{sec:Cerf}.
Let us finish by a last example, which is very simple but interesting for the discussion that follows.
\medbreak\noindent
{\em Example 4.} Let $\mathcal E_n$ be the event that the box $\Lambda_N$ is
  connected to $\{n\}\times\{-n,\dots,n\}^{d-1}$, with $N=N(\delta)>0$ independent of $n$. Here,
$\Longleftarrow$ follows easily from the ergodicity of $\mathbb P_p$ and FKG but again the $\Longrightarrow$ implication 
seems difficult to obtain.  
\medbreak
The search for  a good sequence of events $\mathcal E_n$ has been at the heart of attempts to prove the conjecture. An important development was made in \cite{GriMar90}. In this paper, the authors considered the sequence of events $\mathcal
E_n$ defined in the fourth example. As mentioned above, the $\Longrightarrow$ seems extremely difficult to derive. Nevertheless,
Grimmett and Marstrand introduced a clever renormalisation scheme allowing to
prove the following weaker version of the implication: for any $\ep>0$, there exists
$\delta>0$ such that for every $n$,
\begin{equation} \label{eq:P2ep}
\mathbb P_p[\mathcal E_n]>1-\delta\quad\Longrightarrow\quad\mathbb P_{p+\ep}[0\lr[{\rm Slab}^d_n\,]\infty]>0
\end{equation}
where ${\rm Slab}_n^d\coloneqq\Z^2\times\{-n,\dotsc,n\}^{d-2}$. In words, the implication can be proved if one allows some {\rm sprinkling}. As suggested in \cite{GriMar90}, if one could get rid of the sprinkling by
$\ep$ in the previous statement, then the conjecture would follow.

The goal of this paper is to prove a quantitative version of the Grimmett-Marstrand argument by bounding the critical point of ${\rm Slab}^d_n$ in terms of $n$. In the language of the Grimmett-Marstrand theorem, we will be interested in how small $\ep$ can be taken as a function of $n$. We believe that improving how small $\ep$ can be taken is a good intermediate problem for the conjecture. Getting bounds is non-trivial and requires some understanding of the critical phase. As a consequence, each improvement
on the existing  bounds should shed a new light on the critical behaviour.

There is a quantity which is intimately related to $p_c({\rm Slab}^d_n)$, called the {\em correlation length}, which appears repeatedly in physics. In order to have a statement which is independent of the Grimmett-Marstrand theorem, we choose to first state our main result in terms of the correlation length.

\subsection{An upper bound on the correlation length}

For $p<p_c$, the probability $\mathbb P_p[0\leftrightarrow \partial\Lambda_n]$ decays exponentially fast in $n$ (see \cite{AizBar87,Men86,DumTas15a}). 
The rate at which this happens is known as the {\em correlation length} $\xi_p$, namely
$$\xi_p:=\varlimsup_{n\rightarrow\infty}-\frac{n}{\log\mathbb P_p[0\leftrightarrow \partial\Lambda_n]}.$$
For $p>p_c$, the correlation length is also defined, but the formula is slightly modified:
 $$\xi_p:=\varlimsup_{n\rightarrow\infty}-\frac{n}{\log\mathbb P_p[0\leftrightarrow \partial\Lambda_n,0\not\leftrightarrow\infty]}.$$
Again, the probability decays exponentially fast so $\xi_p$ is finite. This is due to the following. Grimmett and Marstrand \cite{GriMar90} showed that for any $p>p_c$, there exists $n\ge1$ such that
$$\mathbb P_p\big[0\lr[{\rm Slab}^d_n\,]\infty\big]>0.$$
And the exponential decay follows from that by the results of 
\cite{MR905331}. Let us mention that in fact, both limits exist. We could not find a reference for this fact, but it follows using standard methods, see e.g.~\cite[Section 6.2]{Gri99a} for the subcritical case and \cite{MR1048927} for the supercritical case (both prove the existence of the limit with a different definition of the correlation length, but the proofs work also with our definition and the values are equal).

  Our main result is the following.
\begin{theorem}
\label{thm:characteristic-length}
Let $d\ge3$. There exists $C=C(d)>0$ such that for any $p\ne p_c$, 
$$\xi_p\le \exp(C|p-p_c|^{-2}).$$
\end{theorem}
The results below and above $p_c$ are different in nature (even though the same proof gives both), a point which will become clearer when we discuss the proof in the next section. In particular, the use of \cite{GriMar90} to connect slabs and the correlation length mentioned above is used only for $p>p_c$.

Our bound on $\xi_p$ is far from the truth. Conjecturally, one has
$\xi_p=|p-p_c|^{-\nu+o(1)},$
where $o(1)$ tends to $0$ as $p\rightarrow p_c$ ($p\ne p_c$) and $\nu$ is given by 
$$\nu=\begin{cases}\tfrac43&\text{if $d=2$},\\ 0.87\dots&\text{if $d=3$},\\ 0.69\dots&\text{if $d=4$},\\ 0.56\dots&\text{if $d=5$},\\ \tfrac12&\text{if $d\ge6$}.
\end{cases}$$
Some physics references are \cite{AMAH90,LS18,W}. The predictions for $d=3,4,5$ are numerical, while the prediction for $d=2$ is based on conformal field theory, quantum gravity or Coulomb gas formalism, and the prediction for $d\ge6$ on the fact that the model should have a mean-field behaviour.
For site percolation on the triangular lattice, $\xi_p=|p-p_c|^{-4/3+o(1)}$ was proved in \cite{SmiWer01} using the conformal invariance of the model proved in \cite{Smi01}, the theory of Schramm-L\"owner evolution and scaling relations obtained by Kesten in \cite{Kes87b} (such scaling relations were proved under the hyper-scaling hypothesis \cite{BorChaKes99} which is expected to be valid for $d\le 5$). In fact, Russo-Seymour-Welsh theory \cite{Rus78,SeyWel78} combined with \cite{Kes87b} imply that there exists $C>0$ such that $\xi_p\le |p-p_c|^{-C}$ for Bernoulli bond percolation on $\Z^2$. For $d\ge19$, $\xi_p=|p-p_c|^{-1/2+o(1)}$ was proved in \cite{HarSla90,H90} for $p<p_c$. Let us remark that \emph{lower} polynomial bounds may be achieved. We could not found a proof in the literature for this fact, so we include a proof sketch in \S \ref{sec:lower}.


\subsection{A quantitative Grimmett-Marstrand theorem}

The theory of static renormalisation, developed throughout the eighties \cite{AizDelSou80, MR905331, MR1048927, BarGriNew91, GriMar90}, allows to relate the correlation length, percolation in slabs, and various events of the type discussed in \S\ref{sec:crit}. It is a deep theory and we will not attempt to survey it here. But, as already explained, it motivates us to state a version of our main theorem in terms of slabs.
\begin{definition}
  Throughout the paper we denote by $p_n$ the smallest $p<p_c$ such that $\xi_{p}=n$. 
\end{definition}
\begin{theorem}\label{thm:GM}
Fix $d\ge3$. There exists a constant $C=C(d)>0$ such that for every $n\ge3$, \begin{equation}
\bbP_{p_n+\frac{C}{\sqrt{\log
      n}}}[0\lr[{\rm Slab}^d_n\,]\infty]\ge \frac1{2\sqrt{\log
    n}},\label{eq:2}
\end{equation}
In particular, we have that 
$$p_c({\rm Slab}^d_n\,)<p_c+\frac{C}{\sqrt{\log n}}.$$
\end{theorem}

Let us explain the main elements in the proof of Theorems 
\ref{thm:characteristic-length} and  \ref{thm:GM} (they share 95 percent of the proof). The proof is composed of the following 3 steps, each of which requires
to increase the probability somewhat.

\noindent {\bf Step 1}. The result of Chayes and Chayes \cite{CC86} stating that $\mathbb{P}_{p_{c}+\varepsilon}[0\leftrightarrow\infty]\ge\varepsilon$
(\cite{AizBar87} was the first unconditioned proof). We take the opportunity
to give a new proof of this inequality, based on the ideas of \cite{DumTas15a}.

\noindent {\bf Step 2}. The result of Kahn, Kalai and Linial (\cite{KKL88}, see also \cite{Rus82, BouKahKal92, Tal94}) that any boolean function
with small individual influences has at least logarithmic total influence.
We apply this to the function $\mathbbm{1}\{\Lambda_{n}\leftrightarrow\infty\}$
for some $n$. To show that all individual influences are small (as
$n\to\infty)$ we use a geometric argument connecting the probability that a certain edge is pivotal to the same probability for nearby edges. Thus from the information that $\mathbb{P}_{p_{c}+\varepsilon}[0\leftrightarrow\infty]\ge\varepsilon$
we can get $\mathbb{P}_{p_{c}+2\varepsilon}[\Lambda_{n}\leftrightarrow\infty]\ge1-\delta$
(with appropriate connections between the parameters $\varepsilon$, $n$
and $\delta$).

\noindent {\bf Step 3.} A ``seedless'' renormalisation scheme, based on ideas of \cite{2013arXiv1312.1946M}. In Grimmett-Marstrand the renormalisation
follows by finding seeds, i.e.~small boxes (say of size $n$) all whose
edges are open, which are on the boundary of a much larger box, say
of size $N$ \cite{GriMar90}. A first version of our argument which used the same
scheme gave $p_{c}(\textrm{Slab}_{n})<p_{c}+C/\log\log\log\log n$.
Here the path that already exists inside the $N$-box is used in place
of the seeds, each piece of it, if sufficiently separated, can be
used independently. Proposition \ref{prop:uniqueness}
plays a crucial role in the argument. Sprinkling is used as in \cite{GriMar90}, so eventually
we get a renormalisation scheme at $p_{c}+3\varepsilon$. 

The value $\eps=1/\sqrt{\log n}$ comes from the interaction of Steps 2 and
3. In Step 3 we do an $\eps$-sprinkling and the proof requires connections happening with probability at least $1-\exp(1/\varepsilon)$. This forces the $\delta$ of Step 2 to be smaller than $\exp(-1/\varepsilon)$.
But this forces $n$ to be $\exp(1/\varepsilon^{2})$ since our
estimate of the total influence is only logarithmic in $n$. Thus,
the use of \cite{KKL88} is the main constraining factor.

Finally, let us remark on the subcritical case in Theorem \ref{thm:characteristic-length}, i.e.\ on the bound of $\xi_p$ for $p<p_c$. It is a corollary from the supercritical result. There are various ways to perform this conclusion, but here the simplest was simply not to start all the process (i.e.~Steps 1--3) from $p_c$ but rather from an appropriate $p<p_c$ where $\xi_p$ is sufficiently large. Thus we prove both results in one fell swoop. Let us stress again, though, that it is the supercritical result which is central and the subcritical result is merely a corollary.

In \S\ref{sec:lower-bound-prob}-\ref{sec:seedless} we detail these steps, one step per section. In the last sections we prove Theorems \ref{thm:characteristic-length} and \ref{thm:GM}, as well as Proposition \ref{prop:uniqueness}.

\paragraph{Acknowledgments} This work is based on earlier unpublished results achieved with Vladas Sidoravicius, we wish to take this opportunity to thank Vladas for his contributions. Many readers of the first version made important suggestions and corrections, including Rob van den Berg, Raphael Cerf, Roger Silva and an anonymous referee. We thank Ariel Yadin for asking what is known about lower bounds.

HDC was supported by the ERC CriBLaM, a chair IDEX from Paris-Saclay, a grant from the Swiss NSF and the NCCR SwissMap also funded by the Swiss NSF. GK was supported by the Israel Science Foundation, by Paul and Tina Gardner and by the Jesselson Foundation. VT was supported by the NCCR SwissMap, funded by the Swiss NSF.

\section{Preliminaries}\label{sec:prelim}
Fix an integer $d\ge 2$. Two vertices $x$ and
$y$ of $\Z^d$ are said to be {\em neighbours} (denoted $x\sim y$)
if $\|x-y\|_2=1$. In such a case, $\{x,y\}$ is called an {\em edge} of
$\Z^d$. The set of edges is denoted by $E(\Z^d)$. For $n\ge1$, introduce the box
$\Lambda_n:=\{-n,\ldots,n\}^d$ and its (vertex) boundary $\partial
\Lambda_n\coloneqq\Lambda_n\setminus\Lambda_{n-1}$.
Also, we define ${\rm Slab}^d_n:=\Z^2\times\{-n,\dots,n\}^{d-2}$.

A percolation configuration $\omega=(\omega(e):e\in E(\Z^d))$ is an element of $\{0,1\}^{E(\Z^d)}$. If $\omega(e)=1$, the edge $e$ is said to be {\em open}, otherwise it is said to be {\em closed}. Let $S\subset \mathbb Z^d$. Two vertices $x$ and $y$ are said to be {\em connected in $S$} (in $\omega$) if there exists a path $x=v_0\sim v_1\sim v_2\sim\dots\sim v_k=y$ of vertices in $S$ such that $\omega(\{v_i,v_{i+1}\})=1$ for every $0\le i<k$. Let $A$ and $B$ be two subsets of $S$, we write $A\lr[S] B$ if some vertex of $A$ is connected in $S$ to some vertex of $B$,  and $A\lr[S]\infty$ if $A\lr[S]\partial \Lambda_n$ holds for any $n\ge 1$. If $S=\mathbb Z^d$, we drop it from the notation and simply write $A\lr B$ and $A\lr\infty$. A {\em cluster} is a maximal set of vertices that are connected together in $\omega$. 

For $p\in[0,1]$, consider the
Bernoulli bond percolation measure $\bbP_p$ on $\{0,1\}^{\mathbb E(\Z^d)}$ under which the variables $\omega(e)$ with $e\in E(\Z^d)$ are i.i.d.\@ Bernoulli random variables with parameter $p$. Define $p_c=p_c(d)\in(0,1)$ such that
$\mathbb P_p[0\leftrightarrow\infty]$ is $0$ when $p<p_c$ and strictly positive when $p>p_c$. See \cite[Theorem 1.10]{Gri99a} for the fact that indeed $0<p_c<1$.

An event $A$ is {\em increasing} if it is stable to opening edges. The FKG inequality states that increasing events are positively correlated, see \cite[Theorem 2.2]{Gri99a}. The edge $e$ is {\em pivotal for the event $A$} if the configurations $\omega^e$ and $\omega_e$ defined by
$$\omega^e(f)=\begin{cases}\omega(f)&\text{ if }f\ne e\\ 1&\text{ if }f=e,\end{cases}\qquad\text{ and }\omega_e(f)=\begin{cases}\omega(f)&\text{ if }f\ne e\\ 0&\text{ if }f=e\end{cases}$$
satisfy $\omega^e\in A$ and $\omega_e\notin A$.

We will denote by $c$ and $C$ arbitrary constants which depend only on the dimension $d$ (and occasionally other parameters, which will be noted). Their value may change from formula to formula, and even inside the same formula. Occasionally we will number them for clarity. We will use $c$ for constants which are sufficiently small and $C$ for constants which are sufficiently large.

\section{The result of Chayes and Chayes}
\label{sec:lower-bound-prob}

In this section, we prove the following (recall that $p_n$ is the smallest $p<p_c$ such that $\xi_{p}=n$).
\begin{proposition}\label{lem:oneArm}
  For every $n$ large enough,  
  \begin{equation}\label{eq:3}
    \mathbb P_{p_n + \frac{1}{\sqrt{\log n}}}[0\lr \partial\Lambda_{n}]\ge \frac{1}{\sqrt{\log n}}.
  \end{equation}
\end{proposition}
(For the purpose of the supercritical result it would have been enough to know this at $p_c+1/\sqrt{\log n}$, which is exactly the original result of Chayes and Chayes, but for the subcritical result we want to know it with $p_c$ replaced by $p_n$.)
%
%
\begin{proof}
Given a finite set $S$ containing $0$, and a parameter $p\in[0,1]$, define
\begin{equation}\label{eq:5}
  \varphi_p(S):=\sum_{ \substack{x \sim y \\x\in S,\, y\notin S}}p\,\mathbb P_p[0\lr[S]x].
\end{equation}
Fix $n\ge1$. Let us recall two relations between this quantity and the one-arm probability, established in \cite{DumTas15a}. First, for every $S\subset \Lambda_n$ containing $0$, the last displayed equation of Section 2.1 of \cite{DumTas15a} gives the upper bound
\begin{equation}
\mathbb P_p[0\leftrightarrow \partial\Lambda_{nk}]\le \varphi_p(S)^{k-1}.\label{eq:6}
\end{equation}
for every $k\ge1$. Also, the quantity $\varphi_p(S)$ can be used to bound the derivative of the one-arm probability. Lemma 2.1 of \cite{DumTas15a} states that for every $p\in [0,1]$,
\begin{equation}\label{eq:7}
\frac{d}{dp}\mathbb P_p[0\leftrightarrow\partial\Lambda_n]\ge
  \frac1{p(1-p)}\cdot\Big[\inf_{0\in S\subset\Lambda_n}\varphi_p(S)\Big]\cdot(1-\mathbb P_p[0\leftrightarrow\partial\Lambda_n]).
\end{equation}

The proof of Proposition~\ref{lem:oneArm} can be easily derived from the two equations above. If for some $p\in[0,1]$, there exists a subset $S$ of $\Lambda_n$ with $\varphi_p(S)<\frac1e$, then one deduces immediately from \eqref{eq:6} that
$$\mathbb P_p[0\leftrightarrow \partial\Lambda_{k}]\le e^{-A\lfloor k/n\rfloor -1},\qquad A>1,$$
which implies that $\xi_p< n$. As a consequence, $\varphi_{p_n}(S)\ge\frac1e$ for any set $S$ included in $\Lambda_n$ containing $0$. Since $\varphi_p(S)$ is increasing in $p$, we have $\varphi_p(S)\ge \frac1e$ for any $p\ge p_n$, and the differential inequality~\eqref{eq:7} gives that for every $p\ge p_n$,
\begin{equation}\frac{d}{dp}\mathbb P_p[0\leftrightarrow\partial\Lambda_n]\ge \tfrac{4}e(1-\mathbb P_p[0\leftrightarrow\partial\Lambda_n]).\label{eq:8}
\end{equation}
Now, set $p_n':=p_n+1/\sqrt{\log n}$. Either $\mathbb P_{p_n'}[0\leftrightarrow\partial\Lambda_n]> 1-\tfrac{e}4$, or integrating \eqref{eq:8} between $p_n$ and $p'_n$ gives \eqref{eq:3}. This concludes the proof.
\end{proof}
%
\section{Sharp threshold}
\label{sec:sharp-thresh-prob}

In this section we prove the following result.
\begin{proposition}\label{lem:Sharp}
  For every $0<\beta<1$, there exists $C=C(\beta,d)>0$ such that for every $n$
  \begin{equation*}
    \mathbb P_{p_n + C/\sqrt{\log n}} [\Lambda_{n^\beta}\leftrightarrow \partial\Lambda_n]\ge 1-e^{-\sqrt{\log n}}.
  \end{equation*}
\end{proposition}

We may assume without loss of generality that $\beta<\alpha$, where $\alpha$ is chosen such that the statement of Proposition~\ref{prop:uniqueness} holds. Further, we may assume $n$ to be sufficiently large, as for small $n$ and large $C$ we would have $p_n+C/\sqrt{\log n}>1$, making the claim trivial. Set $m:=\lfloor n^\beta\rfloor$ for brevity.  The proposition will follow, using standard arguments, once we prove the following lemma.  
\begin{lemma}\label{lem:pivotal}With $n$, $\alpha$, $\beta$ and $m$ as above, and for any $p$,
\begin{equation}
  \label{eq:28}
  \mathbb P_p[e\text{ is a closed  pivotal for  } \Lambda_m\leftrightarrow\partial\Lambda_n ]\le\frac1{m^{\alpha/4}}.
\end{equation}
\end{lemma}


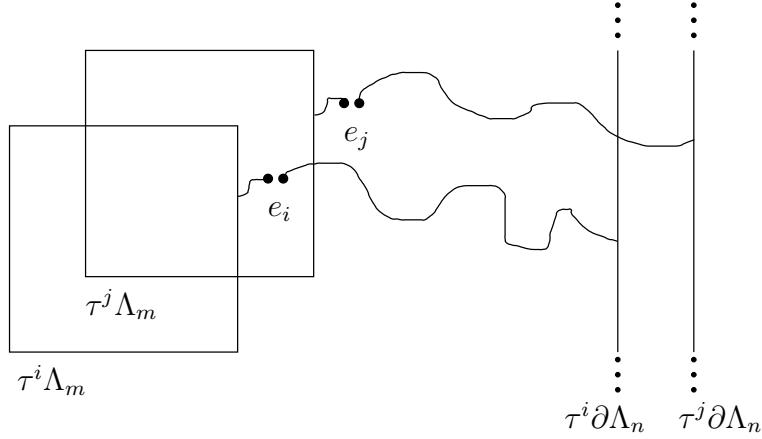
\begin{figure}
\centering
\input{tau.pspdftex}
\caption{Two edges $e_i$, $e_j$ for $i<j$, and the corresponding translated boxes.}\label{fig:tau}
\end{figure}

\begin{proof}
%
Fix an edge $e\in E$ and distinguish between two cases, depending on whether the edge $e$ is close to $\partial \Lambda_m\cup \partial \Lambda_n$ or not. Write $\rho$ for the $L^\infty$-distance between the edge $e$ and $\partial \Lambda_m\cup \partial \Lambda_n$.

If $\rho\ge m^{1/4}$, then observe that a translated version of the  event $A_2(m^{\alpha/4},m^{1/4})$ must occur around the edge $e$ when the edge is a closed pivotal. Therefore, Proposition~\ref{prop:uniqueness} implies that \eqref{eq:28} holds.

The more difficult case is when $\rho\le m^{1/4}$. Let us first assume that $e$ is at a distance smaller than $m^{1/4}$ of $\Lambda_m$. Then there exists a translation $\tau$ by a vector in $\Lambda_{ m^{1/4}}$ such that $e$ belongs to  the translate $\tau \Lambda_m$ of $\Lambda_m$ by $\tau$, and further, such that $e\in\tau^i\Lambda_m$ (where $\tau^i$ denotes the $i$-th iterate of $\tau$) for every $i\in\{1,\dotsc,I\}$, where $I\coloneqq\lfloor \tfrac12 \:m^{3/4}\rfloor$. For $0\le i<I$ define the edges $e_i=\tau^i e$. It follows that for every $i<j$, both endpoints of the edge  $e_i$ belong to   $\tau^j\Lambda_m$, see Figure \ref{fig:tau}. Define also the event $$B_i:=\{e_i\text{ is a closed pivotal for }\tau^i \Lambda_m\leftrightarrow\tau^i \partial \Lambda_n\}.$$  Writing $M$ for the number of indices $i$ for which $B_i$ occurs, translation invariance and the Cauchy-Schwarz inequality imply
\begin{equation}
(I\cdot \mathbb P_p[B_0])^2 =\mathbb E_p[M]^2\le\mathbb E_p[M^2]=\mathbb E_p[M]+ 2\sum_{i<j}\mathbb P_p[B_i\cap B_j].\label{eq:32}
\end{equation}
Let us bound probabilities on the right-hand side. Fix $i<j$ and assume that $B_i\cap B_j$ occurs. Then, we claim that there must exist two disjoint clusters in $\Lambda_{n/2}$ crossing the annulus between $\Lambda_{2m}$ and $\Lambda_{n/2}$. Indeed, one extremity $x_i$ of $e_i$  must be connected to the boundary of $\tau^i \Lambda_{n}$, and one extremity $x_j$   of $e_j$  must be connected to the boundary of $\tau^j \Lambda_{n}$. The fact that $e_j$ is a {\em closed} pivotal implies in particular that $\tau^j\Lambda_m\nleftrightarrow \tau^j\partial\Lambda_n$ and hence, since  $x_i$  belongs to  $\tau^j\Lambda_m$, it is not connected to the boundary of $\tau^j \Lambda_{n}$ so that the clusters of $x_i$ and $x_j$ in the box $\Lambda_{n/2}$ must be disjoint (see again Figure \ref{fig:tau}). For $n$ large enough, we have $2m\le (n/2)^\alpha$ and Proposition~\ref{prop:uniqueness} implies that
\begin{equation}
  \label{eq:31}
  \mathbb P_p[B_i\cap B_j]\le \frac1{2m}. 
\end{equation}
Plugging this estimate in~\eqref{eq:32} and using the trivial bound $M\le I$, we obtain
\begin{equation}
  \label{eq:33}
  \mathbb P_p[B_0]^2\le \frac 1I + \frac 1 m\le \frac 1{\sqrt m},
\end{equation}
provided $n$ is large enough. This completes the proof in this case.

The exact same reasoning also works if one assumes that the edge $e$ is within distance  $m^{1/4}$ of the boundary of $\Lambda_n$. Consider a translation $\tau$ by a vector in $\Lambda_{ m^{1/4}}$ such that  $e$ does \emph{not} belong to  $\tau \Lambda_n$. One can define the edges $e_i$ and the events $B_i$ as above. In this case, for $i<j$, the edge $e_i$ does not belong to $\tau^j \Lambda_n$ and the same reasoning as above concludes the proof.      
\end{proof}
\begin{proof}[Proof of Proposition~\ref{lem:Sharp}]
We use  the  following  standard  sharp threshold result for Boolean functions   (see e.g.~\cite[Corollary 1.2]{Tal94}): for any $\delta>0$, there exists a constant $c'=c'(\delta)>0$ such that for any
increasing event $A$
depending on a finite set $E$ of edges, and any $p\in[\delta,1-\delta]$, 
\begin{equation}\label{eq:diff1}\frac{d}{dp}\Pp{A}\ge c'\log \Big(\frac1{\max\{\Pp{e\text{ pivotal for }A}:e\in E\}}\Big)\cdot \Pp{A}(1-\Pp{A}).\end{equation}
We apply \eqref{eq:diff1} to the event $A=\{\Lambda_m\leftrightarrow \partial \Lambda_n\}$, bounding the pivotality probability inside the log using Lemma \ref{lem:pivotal}. Note that we use here the fact that pivotality is independent of the status of the edge, hence the probability of being \emph{closed} pivotal (which is what we get from Lemma \ref{lem:pivotal}) is $1-p$ times the probability of being pivotal, as needed in \eqref{eq:diff1}. We get that for any $\delta>0$, there exists $c=c(\delta,\beta,d)>0$ such that for every 
  $p\in [\delta,1-\delta]$ and every $n$ large enough,
  \begin{equation}
    \label{eq:10}
    \frac{f'(p)}{f(p)(1-f(p))}\ge c \log n,\quad\text{where }f(p)=\Pp{\Lambda_{n^\beta} \lr \partial\Lambda_n}.
  \end{equation}
Set $p_n':=p_n+\frac 1{\sqrt{\log n}}$. By Proposition~\ref{lem:oneArm},
\begin{equation}
  \label{eq:11}
  \Pp[p_n']{\Lambda_{n^\beta}\lr \partial\Lambda_n}\ge \Pp[p_n']{0\lr \partial\Lambda_n}\ge \frac1{\sqrt{\log n}}.
\end{equation}
We now integrate the differential inequality \eqref{eq:10}.
Define $p_n''$ by $f(p_n'')=\frac12$ and then throughout the interval $[p_n',p_n'']$ we can remove the factor $1-f(p)$ from the denominator, and pay only by halving the constant on the right hand side of \eqref{eq:10}. This gives $\log(f)'\ge c\log n$ which we integrate and get that $p_n''$ must be no more than $p_n'+C\log\log n/\log n$. Similarly, in the interval $[p_n'', p_n''+C/\sqrt{\log n}]$ we remove the factor $f(p)$ from the denominator, get $-\log(1-f)'\ge c\log n$ and arrive at the conclusion that $f(p_n''+C/\sqrt{\log n})\ge 1-\exp(-\sqrt{\log n})$, if $C=C(\beta,\delta)$ is sufficiently large. This is exactly the conclusion of the proposition (with a larger $C$ to compensate for replacing $p_n''$ with $p_n$).
\end{proof}

\begin{remark}
Proposition~\ref{lem:Sharp} can be directly obtained using Section~3 of \cite{DRT17} with the definition of the event $A_k$ being, for $0\le k\le \tfrac12n^\beta$, 
$$A_k=A_k(n):=\{\Lambda_{n^\beta/2+k}\longleftrightarrow\partial\Lambda_n\}.$$ 
Roughly speaking, since all the events $A_k$ have a probability larger than $1/\sqrt{\log n}$ at $p_n$, the argument in \cite{DRT17} implies that at every $p\ge p_n$, one of the event $A_k$ has a logarithmic derivative larger than $c\log n$ for some small constant $c$. A careful manipulation enables to prove that one of the events $A_k$ (and therefore $A_{n^\beta/2}$) must have probability larger than $1-e^{-\sqrt{\log n}}$ at $p_n+C/\sqrt{\log n}$. We believe that the present solution is simpler in the case of Bernoulli percolation and may have further applications, even though the other alternative does not use the Aizenman-Kesten-Newman estimate on the probability of the two-arm event.
\end{remark}
\section{The seedless renormalisation scheme}\label{sec:seedless}
\def\conda{\texorpdfstring{\protect\hyperlink{cond:a}{\bf (a)}}{(a)}\xspace}
\def\condb{\texorpdfstring{\protect\hyperlink{cond:b}{\bf (b)}}{(b)}\xspace}
\def\condc{\texorpdfstring{\protect\hyperlink{cond:c}{\bf (c)}}{(c)}\xspace}
The normalisation scheme we will work with uses four different scales, which we will denote by $k<K<n<N$. The most important is the scale between $K$ and $n$, where we will insert $1/\ep^2$ boxes of size $K$ and use the independence between these boxes to get to an event with high probability. The scales between $k$ and $K$; and between $n$ and $N$ will be used for gluing paths using Proposition \ref{prop:uniqueness} (the second one, between $n$ and $N$, is used only for resolving a technical issue of connecting to a specific facet and is less important). Here is the exact formulation which we will use.
\begin{theorem} \label{thm:quantitative GM} Fix $d\ge 3$. There exists a constant $C=C(d)>0$ such that the following holds. Assume that for some $p\in[0,1]$ and some $\ep>0$, there exist $1 \le k \le K\le n\le N<\infty$ such that $K\le \ep^2 n$ and 
  \begin{enumerate}
  \item[\hypertarget{cond:a}{\bf (a)}] $\Pp{0\lr\partial \Lambda_N}\ge \ep$,
  \item[\hypertarget{cond:b}{\bf (b)}] $\Pp{\Lambda_k \lr \partial\Lambda_N}\ge1- \exp(-\tfrac{1}{\ep})$,
  \item[\hypertarget{cond:c}{\bf (c)}] $\Pp[p]{A_2(k,K)} \le  \exp(-\tfrac{1}{\ep})$ and $\Pp[p]{A_2(n,N)} \le  \exp(-\tfrac{1}{\ep})$.
  \end{enumerate} 
  Then
  \begin{equation*}
    \mathbb P_{p+C\ep}[0\lr[{\rm Slab}^d_{2N}\,]\infty]\ge \tfrac\ep2.
  \end{equation*}

\end{theorem}

Again, the reader who is interested only in the supercritical case may mentally replace ``some $p\in [0,1]$'' with ``$p_c+2\eps$'' as in the proof sketch in the introduction. But for the subcritical result we will apply it at $p+2\eps$ for a slightly subcritical $p$, and we do not know, eventually, if $p+2\eps$ is sub- or supercritical.


\def\condbp{\texorpdfstring{\protect\hyperlink{cond:bp}{\bf (b')}}{(b')}\xspace}
The proof is divided into two parts. In the first one, we prove the following intermediate statement.
\begin{lemma}\label{lem:bprime}
  Assume that conditions \conda, \condb and \condc hold. Then, there exists some $c>0$ (depending on $d$ only) such that for every connected set $S\ni0$ with a diameter larger than $n$, $$\Pp{S\lr F(N)}\ge1-2\exp[-c/\ep],$$ where $F(N):=\{(x_1,\ldots,x_d)\in\partial \Lambda_{N}\::\: x_1=N,x_2\ge 0,\ldots,x_d\ge 0\}.$
\end{lemma}
Here and below we call sets such as $F(N)$ ``quarter-faces'' even though this name is correct only in $d=3$.
\begin{remark}
The introduction of quarter-faces is a purely technical step and should not worry the reader. Indeed, the probability of connecting to a quarter-face is easily compared to the probability of connecting to the boundary of the box. To this end, divide $\partial\Lambda_N$ into $d 2^d $ quarter-faces $F_1,\ldots,F_{d2^d}$. Using the Harris-FKG inequality (sometimes called ``the square root trick'' when used in this way, see \cite[equation (11.14)]{Gri99a}) together with \condb, we find that
\begin{align}
  \label{eq:12}
  \bbP_p[\Lambda_k\lr[\Lambda_{N}] F(N)] \ge 1-\exp[-1/(\ep d2^d)].
\end{align}
\end{remark}
  The conclusion of Lemma \ref{lem:bprime} can be understood as a strengthening of the condition \condb where the box $\Lambda_{k}$ is replaced by arbitrary sufficiently large sets, and the boundary of $\Lambda_N$ is replaced by one of its quarter-faces. Using it, we will be able to construct an infinite cluster in ${\rm Slab}^d_{2N}$ by propagating it using local connections. Heuristically, if the cluster of the origin is connected to a large box $\Lambda$ away from $0$, then it must contain a large set, which is sufficient to propagate this cluster to other boxes neighbouring $\Lambda$. 
The condition on connectedness of arbitrary large sets was introduced in the work of Martineau and Tassion \cite{2013arXiv1312.1946M}, where it was established using abstract measurability arguments. The main contribution here is to make it quantitative.

  

The proof of Theorem \ref{thm:quantitative GM} is now organised as follows. In \S~\ref{sec:proof-conditionbf-b}, we prove Lemma \ref{lem:bprime}. 
The proof of the main theorem is then concluded in \S~\ref{sec:proofFC}. 

\subsection{Connections to arbitrary sets}
\label{sec:proof-conditionbf-b}
In this section we prove Lemma \ref{lem:bprime}. Without loss of generality we may assume $\ep$ is sufficiently small (as otherwise by choosing $c$ sufficiently small the claim is trivially true). 
Below, the constants $c_i$ depend on $d$ only.

Let $p\in[0,1]$, $\ep>0$ and $k\le K\le n\le N$ be such that $K\le \ep^2 n$ and the three conditions \conda, \condb and \condc hold. Fix a connected set $S$ containing $0$ with a diameter at least $n$. Without loss of generality, we may assume $S\subset \Lambda_n$.

Consider a family of points $x_1,\dots,x_\ell\in S$ such that the boxes $Q''_i:=x_i+\Lambda_K$ are all disjoint and included in $\Lambda_n$. Also, introduce the smaller box $Q'_i:=x_i+\Lambda_k$. Note that we may choose 
$\ell\ge c_1/\ep^2$ such points, so let us fix $\ell=\lceil c_1/\eps^2\rceil$.
 
For every $i\in\{1,\dots,\ell\}$, define the two events
\begin{align*}E_i&:=\{x_i\longleftrightarrow \partial Q''_i\}\cap \{\exists\text{ unique cluster in $Q''_i$ from $Q'_i$ to $\partial Q''_i$}\},\\
  B_i&:=\{Q'_i\nlr \partial\Lambda_N \}.
\end{align*}
By translation invariance and conditions \conda and \condc,
\begin{align*}\Pp{E_i}&\ge \Pp{x_i\longleftrightarrow \partial Q''_i}-\Pp{A_2(k,K)}\\
  &\ge \ep-\exp(-1/\ep)\ge\ep/2.
\end{align*}
Since the boxes $Q_i''$ are disjoint, the events $E_i$ are independent, and hence $$\Pp{\cup E_i}\ge 1-2e^{-c_2\ep\ell}\ge1-2e^{-c_3/\eps}.$$
where the second inequality is from our requirement that $\ell\ge c_1/\eps^2$.

Now, \condb implies that for every $i$,
\begin{equation}
\Pp{B_i} \le \exp[-1/(d2^d\ep)].
\end{equation}
Indeed, find a quarter-face $F$ of $x_i+\Lambda_N$ outside $\Lambda_{N-1}$ and then apply \eqref{eq:12} (with 0 shifted to $x_i$) to get
\[
\Pp{B_i}\le\Pp{Q_i'\nleftrightarrow F}
\stackrel{\smash{\textrm{\eqref{eq:12}}}}{\le}\exp[-1/(d2^d\ep)].
\]
By a union bound we have
\begin{align}
  \Pp{\cup B_i}&
  \le \ell\exp[-1/(d2^d\ep)]\le \exp(-c/\eps)
\end{align}
where the last inequality is by our assumption that $\ell\le c_1/\ep^2+1$ and that $\ep$ is sufficiently small.

Assume $E_i\setminus B_i$ occurred for some $i$. Then we know that $x_i\leftrightarrow\partial Q_i''$ (from the first part of $E_i$), that $Q_i'\leftrightarrow \partial\Lambda_N$ (from the negation of $B_i$) and that the two clusters performing these two connections are the same (from the second part of $E_i$). We get that there is a path from $x_i$ to $\partial\Lambda_N$, and in particular from $S$ to $\partial \Lambda_N$. This gives
\begin{align}\Pp{S\lr \partial\Lambda_N}&\ge \Pp{\exists i\textrm{ s.t.\ }E_i\setminus B_i}\\
&\ge \Pp{\cup E_i}-\Pp{\cup B_i}\\
&\ge 1-2e^{-c_4/\ep}.\label{eq:15}
\end{align}
It remains to replace the boundary of $\partial\Lambda_N$ in the equation above by the quarter-face $F(N)$. Yet, because we assumed $S\subset\Lambda_n$,
\begin{equation*}
  \mathbb P_p [S\lr[\Lambda_{N}] F(N)]\ge \mathbb P_p[\{\Lambda_n \lr F(N)\}\cap \{S\lr \partial\Lambda_N\}\cap A(n,N)^c]\ge 1-Ce^{-c_5/\ep}
\end{equation*}
thanks to \condb (again in the form \eqref{eq:12}) and \condc. This concludes the proof.\qed

\subsection{Renormalisation}
\label{sec:proofFC}

To prove Theorem \ref{thm:quantitative GM} we couple a growing exploration process on the slab with a growing exploration process  on a rescaled version of the square lattice. One will need a simple condition for a growing exploration process on $\mathbb Z^2$ to contain an infinite cluster. Therefore, before moving to the proof, we describe a particular type of exploration process on $\mathbb Z^2$ and give a sufficient condition for the existence of an infinite connected component.

Fix an arbitrary ordering of the edges of $\mathbb Z^2$. Let $\{0\}=A_0\subset A_1\subset A_2\ldots$ and $\emptyset=B_0\subset B_1\subset B_2\ldots$ be two growing sequences of subsets of $\mathbb Z^2$. We say that the sequence $X_t=(A_t,B_t)$ is an \emph{exploration sequence} if for every $t\ge 0$,
\begin{align}
&X_{t+1}=X_t\text{ if there is no edge connecting $A_t$ to $(A_t\cup B_t)^c$},\\
&X_{t+1}=(A_t\cup\{x_t\},B_t)\text{ or } X_{t+1}=(A_t,B_t\cup\{x_t\}) \text{ otherwise},\label{eq:16}
\end{align}
where $x_t$ is the endpoint in $(A_t\cup B_t)^c$ of the minimal edge connecting $A_t$ to $(A_t\cup B_t)^c$ (here and below, when we write minimal we mean with respect to the chosen ordering of the edges of $\mathbb Z^2$). A typical example of a random exploration sequence results from the exploration of the cluster of the origin in a site percolation process on $\mathbb Z^2$. In this case, the set $A_t$ corresponds to the open sites discovered after $t$ steps of exploration and $B_t$ is the discovered part of the (closed) boundary of the cluster.

We say that an exploration sequence  \emph{percolates} if the set $\cup_{t\ge 0}A_t$ is infinite. The following lemma, proved in \cite[Lemma 1]{GriMar90}, gives a sufficient condition for a random  exploration sequence  to percolate.
\begin{lemma}\label{lem:criterionPercolation}
  Let $p_c^{\mathrm{site}}$ be the critical parameter of Bernoulli site percolation on $\mathbb Z^2$. Let $X_t=(A_t,B_t)$ be a random exploration sequence and assume that there exists some $q>p_c^{\mathrm{site}}$ such that for every $t\ge 0$,
\begin{equation}
  \label{eq:17}
  \mathbb P\big(B_{t+1}= B_t  \:|\:X_0,\ldots,X_t\big)\ge q\text{ a.s.},
\end{equation}
then the process $X$ percolates with probability larger than a constant $c=c(q)>0$ that can be taken arbitrarily close to 1 provided that $q$ is close enough to 1.
\end{lemma}

We now return to the proof of the theorem.  For every $x\in \mathbb Z^2$, set $ \Lambda_x=Nx+\Lambda_N\subset \mathbb Z^d$ and $\widetilde \Lambda_x=Nx+\Lambda_{2N}\subset\mathbb Z^d$, where for both we identify $x=(x_1,x_2)\in\mathbb Z^2$ with $(x_1,x_2,0,\dotsc,0)\in\mathbb Z^d$. We will identify $0$ with $(0,0)$ so $\Lambda_0$ is the box of size $N$ centered at $0$ in $\mathbb Z^d$.

Let $\omega$ be a Bernoulli percolation of parameter $p$ in ${\rm Slab}^d_{2N}$ and for every $x\in \Z^2$, let $\omega^x$ be a $\lambda\varepsilon$-percolation on $\widetilde \Lambda_x$, where $\lambda$ is some constant to be fixed later. We assume that  $\omega$ and the $\omega^x$'s  are independent of each other. We will prove that the origin is connected to infinity in
\begin{equation}
     \omega_{\mathrm{total}}:=\omega\vee \big(\vee_{x\in\Z^2} \omega^x\big)\label{eq:18}
\end{equation}
with a probability which is larger than $\ep/2$ (the notation $\vee$ stands for the maximum, or the union of the open edges if one prefers). This will 
   conclude the proof since $\omega_{\mathrm{total}}$ is stochastically dominated by a $(p+25\cdot \lambda\varepsilon)$-percolation --- each edge of the slab appears in at most $25$ boxes $\widetilde\Lambda_x$ (note that the number 25 does not depend on $d$ because $x$ is taken only in $\mathbb Z^2$). 

   To prove this claim, define an increasing sequence of percolation configurations $(\omega_t)_{t\ge0}$ in the slab, coupled with a random exploration sequence $X_t=(A_t,B_t)$ in $\mathbb Z^2$. Given a percolation configuration $\omega$ in the slab, let $\mathcal C(\omega)$ be the set of vertices that are connected inside $\mathbb Z^2\times\{-2N,\ldots,2N\}^{d-2}$ to $0$ by a path of $\omega$. 
   \begin{definition}
   Set $X_0=(A_0,B_0):=(\{0\},\emptyset)$ and $\omega_0=\omega$. For every $t\ge0$, let $\omega_{t+1}$ and $X_{t+1}$ be constructed from $\omega_{t}$ and $X_t$ as follows. If there is no edge connecting $A_t$ to $(A_t\cup B_t)^c$, define $X_{t+1}=X_t$. Otherwise, let $x=x_t$ be the extremity in $(A_t\cup B_t)^c$ of the minimal edge connecting $A_t$ to $(A_t\cup B_t)^c$ and define
   \begin{align}
     &\omega_{t+1}:=\omega_t\vee \omega^x,\\
     &X_{t+1}:=
     \begin{cases}
       (A_t\cup\{x\},B_t)&\text{if }0\lr[]\Lambda_x\text{ in }\omega_{t+1},\\
       (A_t,B_t\cup\{x\})&\text{otherwise}.
     \end{cases}
   \end{align}
   \end{definition}
\begin{remark}There is something unorthodox in the exploration process just defined, as we are not constraining the length of the paths that are created in each step. For example, it is possible that the path in $\mathbb Z^d$ that is reponsible to connect $0$ to $\Lambda_{(0,1)}$ goes much further than $N$. This will have no effect on the argument, but the reader who prefers a more ``geometric'' exploration is welcome to change the definition of $X_{t+1}$ to one where the path demonstrating $0\leftrightarrow\Lambda_x$ is constrained in some reasonable explored domain. The proof would go through equally well.
\end{remark}

\noindent Returning to the construction, we have the following two properties:
   \begin{itemize}
   \item[(i)] $\omega_\infty := \vee_{t\ge0}\: \omega_t\le \omega_{\mathrm{total}}$, 
   \item[(ii)] if $(X_t)$ percolates, then $0$ is connected to infinity in $\omega_\infty$.
   \end{itemize}
   We now wish to prove a third property which, when combined with the previous two and \conda, concludes the proof.
   \begin{itemize}   \item[(iii)] $\mathbb P[X\text{ percolates}\:|\:0\leftrightarrow\partial\Lambda_0\text{ in }\mathcal C(\omega_0)]\ge 1/2$.
   \end{itemize}
   
The proof relies on an application of Lemma~\ref{lem:criterionPercolation}. In order to apply this lemma, let us fix $q>p_c^{\mathrm{site}}(\mathbb Z^2)$ in such a way that $c(q)\ge 1/2$ and try to prove
   \begin{equation}
  \label{eq:19}
  \mathbb P\big(B_{t+1}=B_t \:|\:X_0,\ldots,X_t\big)\ge q\ \text{ a.s.}
\end{equation}
   Since $B_{t+1}=B_t$ as soon as there is no edge connecting $A_t$ to $(A_t\cup B_t)^c$, we can focus on the case where the minimal edge $e$ connecting $A_t$ to $(A_t\cup B_t)^c$ is well defined, and therefore its endpoint $x$ in $(A_t\cup B_t)^c$ also is. In this case, we have $B_{t+1}=B_t$ if $0$ is connected to $\Lambda_x$ in $\omega_{t+1}$. Since $X_0,\ldots,X_t$ and the event that $x$ is well defined are measurable with respect to $\mathcal C(\omega_0),\ldots,\mathcal C(\omega_t)$, it suffices to show that for any admissible $C_0,C_1,\ldots,C_t$, we have
\begin{equation*}
  \mathbb P\big(  \Lambda_0 \lr \Lambda_x\text{ in }\omega_{t+1}\:|\:  \mathcal C(\omega_0)=C_0,\ldots,\mathcal C(\omega_t)=C_t    \big)\ge q\ \text{ a.s.},
\end{equation*}
which would follow if we showed that for every admissible $C_t$,
\begin{equation}
  \label{eq:20}
  \mathbb P\big(  C_t \lr  \Lambda_x\text{ in }\omega_{t}\vee \omega^x \:|\:  \omega_t|_{\partial_E C_t}\equiv 0 \big)\ge q.
\end{equation}
\begin{floatingfigure}[r]{5.5cm}
\begin{centering}
\input{halfface.pspdftex}
\end{centering}
\end{floatingfigure}
\vspace{-0.65cm}
\noindent Now, observe that any admissible $C_t$ must intersect $\Lambda_{x'}$, where $x'$ is the endpoint of $e$ in $A_t$. Furthermore, the diameter of $C_t$ must be at least $N$ (here is where we use the conditioning over the event $0\lr\partial\Lambda_0$). Let $y$ be a vertex of $C_t\cap\Lambda_{x'}$. Since at least one of the quarter-faces of $y+\Lambda_{N}$ is included in $\Lambda_x$ (see an example in the figure on the right, where it is denoted by $F$; the reader is kindly requested to imagine the third dimension), Lemma \ref{lem:bprime} (applied after shifting 0 to $y$) implies
\begin{equation}
  \label{eq:22}
  \mathbb P_p[C_t\lr[\widetilde{\Lambda}_x] \Lambda_x\textrm{ in }\omega\vee\omega^x]\ge 1-2e^{-c/\ep}.
\end{equation}
Since the event $C_t\lr\Lambda_x$ is increasing, we may replace $\omega\vee\omega^x$ with $\omega_t\vee\omega^x$.
Since $\omega^x$ is independent of $\omega_{t}$, Lemma~\ref{sec:lemma-1} below shows that \eqref{eq:20} holds, provided the constant $\lambda$ is large enough. This concludes the proof of Item (iii) and therefore of Theorem \ref{thm:quantitative GM}. \qed


For the next (and last) lemma, it will be convenient to have a notation for the edge boundary restricted to a fixed set. Fix therefore a set $R\subset \Z^d$, and define
$$\Delta A=\big\{\{x,y\}\subset R\::\:|x-y|=1,\, x\in A,\, y\in R\setminus A\big\}.$$

\begin{lemma}\label{sec:lemma-1}
  For any $\delta,\eta>0$, there exists $\lambda>0$ such that
  for any $p\in[\delta,1-\delta]$ and $\ep>0$, as well as any $A,B\subset R$,
$\mathbb P_p[A\lr[R]B]\ge 1-2\exp(-\eta/\ep)$ implies that
$$\mathbb P\big[A\lr[R]B\text{ in }\omega\vee\tilde\omega\,\big|\,\omega(e)=0,\forall e\in\Delta A\big]\ge 1-\delta,$$
where $\omega$ is a Bernoulli percolation configuration satisfying $\mathbb P[\omega(e)=1]\ge p$ for every $e$, and $\tilde\omega$ a Bernoulli percolation of parameter $\lambda\ep$ which is independent of $\omega$.
\end{lemma}

\begin{proof}
If $A\cap B\ne \emptyset$, the result is obvious. We therefore assume $A\cap B=\emptyset$. Also, introduce the event $E$ that $\omega(e)=0\text{ for all }e\in \Delta A$ and the set $W$ defined by
$$W=\big \{\{x,y\}\in \Delta A\text{ and }y\lr[R\setminus A]B \text{ in $\omega$}\big\}.$$
Any path from $A$ to $B$ in $R$, open in $\omega$, must use at least one edge of $W$. Consequently, for any $t\in\N$, we have
$$
\mathbb P_p[A\nlr[R] B] \geq {(1-p)}^{t-1} \,\mathbb P_p[|W|< t].
$$
Then, using that $|W|\ge t$ is independent of the event $E$, we deduce that, still for an arbitrary $t$,
\begin{align*}\mathbb P[A\lr[R]B\text{ in }\omega\vee \tilde\omega\:|\:E]
  &\ge \mathbb P[\exists e\in W:\tilde\omega(e)=1,W\ge t\:|\:E]\\
 &\ge (1-(1-\lambda\ep)^t)\mathbb P[W\ge t]\\
 &\ge (1-(1-\lambda\ep)^t)\Big(1-\frac{\mathbb P_p[A\nlr[R] B]}{(1-p)^{t-1}}\Big)\\
 &\ge (1-(1-\lambda\ep)^t)\Big(1-\frac{\exp(-\eta/\ep)}{(1-p)^{t-1}}\Big).
 \end{align*}
Choosing $\lambda=\lambda(\delta,\eta)$ large enough, the result follows by optimizing on $t$.
\end{proof}

\section{Proofs of Theorems \ref{thm:characteristic-length} and \ref{thm:GM}}
Theorem \ref{thm:GM} follows immediately from what was already proved. Indeed, we use Theorem \ref{thm:quantitative GM} with 
\begin{align*}p&=p_n+\lambda/\sqrt{\log n},\\ 
\eps&=1/\sqrt{\log n},\\
(k,K,n,N)&=(n^{\alpha^3},n^{\alpha^2},n^\alpha,n),
\end{align*} 
where $\alpha$ is given by Proposition~\ref{prop:uniqueness}, and where $\lambda$ is some sufficiently large constant. By the definition of $\alpha$, condition \condc of Theorem \ref{thm:quantitative GM} is satisfied when $n$ is large enough. Condition \conda follows from Proposition \ref{lem:oneArm}, while Condition \condb follows from Proposition \ref{lem:Sharp}, if only $\lambda$ is sufficiently large (we use Proposition \ref{lem:Sharp} with $\beta=\alpha^3$).

The $p<p_c$ case of Theorem \ref{thm:characteristic-length} is identical. Given $p<p_c$ we define $n=\lfloor \xi_p\rfloor$ and then $p_n\le p$. Using Theorem \ref{thm:quantitative GM} in the same way and with the same parameters as above gives that at $p_n+C/\sqrt{\log n}$ we already have percolation in a slab, and in particular it is above $p_c$. Hence
\[
p_c\le p_n+\frac{C}{\sqrt{\log n}}\le p+\frac{C}{\sqrt{\log \xi_p}}
\]
which is identical to $\xi_p\le\exp(C(p_c-p)^{-2})$, as claimed.

For the case $p>p_c$ of Theorem \ref{thm:characteristic-length} we need to estimate the probability to percolate in a slab \emph{starting from the boundary of the slab}. It will be slightly more convenient to work in the ``other slab'', $\oslab_n=\{-n,\dotsc,n\}\times\mathbb Z^{d-1}$. Define
\[
\theta(p,n):=\PP_p[(-n,0,\dotsc,0)\lr[\oslab_n]\infty].
\]
To estimate $\theta(p,n)$ we use the fact that at $p_c$
\[
\sum_{x\in\partial\Lambda_n}\PP_{p_c}[0\lr[\Lambda_n]x]\ge c.
\]
(this is well-known, and in fact we already gave a proof of that while proving Proposition \ref{lem:oneArm}). By symmetry the same holds for the bottom face, i.e.
\[
\sum_{x\in{\{-n\}\times\{-n,\dotsc,n\}^{d-1}}}\PP_{p_c}[0\lr[\Lambda_n]x]\ge c.
\]
Hence for some $x$ on this face we have $\PP_{p_c}[0\lr[\Lambda_n]x]\ge cn^{1-d}$. By translation invariance we get for some $y\in \{0\}\times\{-n,\dotsc,n\}^{d-1}$ that
\[
\PP_{p_c}[(-n,0,\dotsc,0)\lr[\oslab_n]y]\ge cn^{1-d}.
\]
Let now $p>p_c$ and define $n$ such that $\PP_p[0\lr[\slab_n]\infty]\ge 1/2\sqrt{\log n}$. By Theorem \ref{thm:GM} we may take $n\le\exp(C(p-p_c)^{-2})$. We may certainly replace $\slab_n$ with $\oslab_n$, as it is larger. We get
\begin{align*}
  \theta(p,n)&\ge \PP_p[(-n,0,\dotsc,0)\lr[\oslab_n] y,y\lr[\oslab_n]\infty]\\
  &\ge \PP_p[(-n,0,\dotsc,0)\lr[\oslab_n] y]\PP_p[y\lr[\oslab_n]\infty]\ge cn^{1-d}\cdot\frac{1}{2\sqrt{\log n}}
\end{align*}where we used FKG, translation invariance and the fact that the event $(-n,0\dotsc,0)\lr y$ is monotone. By \cite[Theorem 5]{MR905331}, $\xi_p\le n/\theta(p,n)$ and Theorem \ref{thm:characteristic-length} is established.\qed

\section{On Proposition \ref{prop:uniqueness}}\label{sec:Cerf}
Proposition \ref{prop:uniqueness} was not stated in this generality in the paper of Cerf \cite{MR3395466} (in that paper, the polynomial upper bound is stated for $p=p_c$). Here we explain how the arguments of \cite{MR3395466} can be adapted  to get a bound which  is uniform in $0\le p\le 1$.

\begin{proof}[Proof of Proposition~\ref{prop:uniqueness}]
  It suffices to prove the estimate above for $p\in (\delta,1-\delta)$, for some fixed  $\delta>0$ small enough. Indeed if $p$ is close to $0$ or $1$, one can easily prove the bounds of the proposition using standard perturbative arguments. Now, using for example the inequality above Proposition 5.3 of \cite{MR3395466}, we see that there exists a constant $\kappa>0$ such that for every $\delta \le p\le 1-\delta$, and every $n\ge 1$
  \begin{equation}
    \label{eq:23}
    \mathbb P_p[A_2(0,n)]\le \frac{\kappa \log n}{\sqrt n}.
  \end{equation}
  We will prove that there exists $C\ge 1$ large enough such that, uniformly in $\delta \le p\le 1-\delta$, we have for every $n\ge 2$
  \begin{equation}
    \label{eq:24}
    \mathbb P_p[A_2(n,n^C)]\le \frac1n,
  \end{equation}
  which concludes the proof. In the proof of \cite{MR3395466}, the fact that $p=p_c$ was used in order to obtain a lower bound on the two-point function (see Lemma~6.1 in \cite{MR3395466}). One can replace the input coming from the hypothesis that $p=p_c$ by the following simple argument. Fix $n\ge 2$. Since $\mathbb P_p[A_2(n,n^C)]\le\mathbb P_p[\Lambda_n \lr \partial \Lambda_{2n}]$, we can assume that the probability that there exists an open path from $\Lambda_n$ to $\partial \Lambda_{2n}$ is larger than $1/n$. Therefore, by the union bound, there must exist a point $x$ at the boundary of $\Lambda_n$ that is connected to $x+\partial \Lambda_n$ with probability larger than $\frac1{n|\partial \Lambda_n|}$. Hence, by translation invariance, $0$ is connected to $\partial \Lambda_n$ with probability larger than $\frac1{n|\partial \Lambda_n|}$, and the union bound again implies that for every $m\le n$,
  \begin{equation}
    \label{eq:25}
    \sum_{y\in \partial \Lambda_m}\mathbb P_p[0\lr[\Lambda_m]y] \ge \frac 1{n|\partial\Lambda_n|}. 
  \end{equation}
  Using this estimate, one can repeat the argument of Lemma~6.1 in \cite{MR3395466} to show that there exists a constant $C>0$ (independent of $n$ and $p$) such that 
  \begin{equation}
    \label{eq:26}
    \forall x,y\in \Lambda_n\quad \mathbb P_p[x\lr[\Lambda_{2n}]y]\ge \frac1{n^C}. 
  \end{equation}
  Then one can conclude the proof using the estimate above together with Corollary 7.3 in \cite{MR3395466}.
\end{proof}

\section{A lower bound}\label{sec:lower}
In this section, we make a few remarks on lower bounds for the correlation length. We first note that \cite{AN84} shows that at $p<p_c$ the expected size of the cluster (``the susceptability'') is at least $1/(p_c-p)$, and this shows that $\xi_p\ge (p_c-p)^{-1/d}$. Newman \cite{N86} shows a lower bound also on the truncated susceptability for $p>p_c$ but he makes assumptions on the behaviour of critical percolation which are still unknown: we could not complete Newman's argument without assuming Conjecture \ref{conj:1}.

Here we will sketch a proof that $\xi_p\ge (p_c-p)^{-2/d+o(1)}$ in the case that $p<p_c$, leaving the more complicated case of $p>p_c$ for the future. Let $N>0$, and let $E$ be the event that there exists an easy-way crossing of the box $3N\times\dotsb\times 3N\times N$. By \cite[\S 5.1]{Kes82} there is a constant $c_1(d)$ such that if $\mathbb P_q[E]<c_1$ then $q<p_c$.

Let now $p<p_c$.
Standard arguments using supermultiplicativity (see \cite{MR905331}) show that for every $x$ with $|x|>N$ we have  $\mathbb P_p(0\leftrightarrow x)\le \exp(-N/\xi_p)$. Hence there exists $N\approx\xi_p\log\xi_p$ such that $\mathbb P_p[E]<\frac12 c_1(d)$. Now, it is well-known that for any boolean function $f$ the total influence $I(f)$ satisfies $I(f)\le\sqrt{n\textrm{var}f/p(1-p)}$. Defining $F(p)=\mathbb E_p[f]$ this gives, for $p$ bounded away from 0 and 1, $\sqrt{F}'\le C\sqrt{n}$. We apply this for $f$ being the indicator of the event $E$ and get
\begin{equation}\label{eq:deriv sqrt}
\frac{d}{dp}\sqrt{\mathbb P_p[E]}\le CN^{d/2}\le C(\xi_p\log\xi_p)^{d/2}.
\end{equation}
Hence at $q\coloneqq p+c_2(\xi_p\log\xi_p)^{-d/2}$ for some $c_2$ sufficiently small, we would have $\mathbb P_q[E]<c_1(d)$ and hence $q<p_c$. The claim follows.

\bibliographystyle{alpha}
\newcommand{\etalchar}[1]{$^{#1}$}

\end{document}

%% file: tau.pspdftex
\begin{picture}(0,0)%
\includegraphics{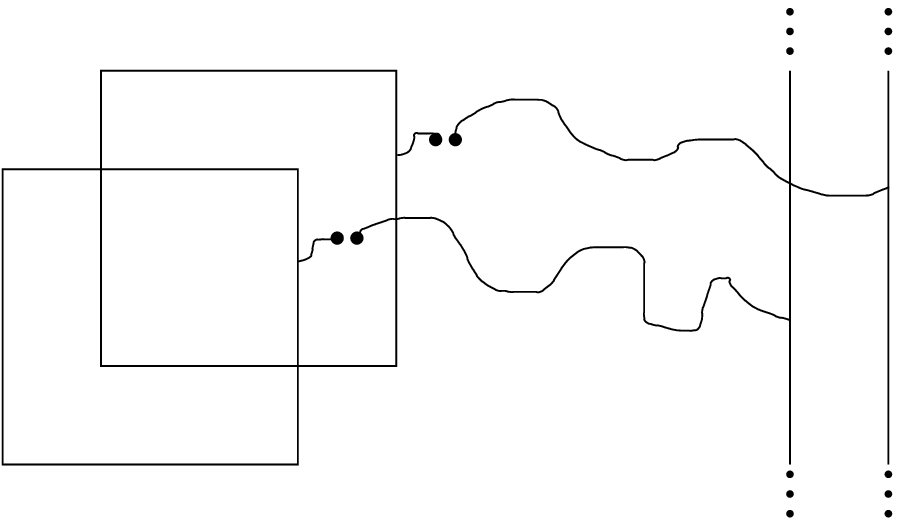}%
\end{picture}%
\setlength{\unitlength}{4144sp}%
\begingroup\makeatletter\ifx\SetFigFont\undefined%
\gdef\SetFigFont#1#2#3#4#5{%
  \reset@font\fontsize{#1}{#2pt}%
  \fontfamily{#3}\fontseries{#4}\fontshape{#5}%
  \selectfont}%
\fi\endgroup%
\begin{picture}(4083,2614)(439,-2834)
\put(1981,-1501){\makebox(0,0)[lb]{\smash{{\SetFigFont{12}{14.4}{\rmdefault}{\mddefault}{\updefault}{\color[rgb]{0,0,0}$e_i$}%
}}}}
\put(496,-2536){\makebox(0,0)[lb]{\smash{{\SetFigFont{12}{14.4}{\rmdefault}{\mddefault}{\updefault}{\color[rgb]{0,0,0}$\tau^i\Lambda_m$}%
}}}}
\put(901,-2086){\makebox(0,0)[lb]{\smash{{\SetFigFont{12}{14.4}{\rmdefault}{\mddefault}{\updefault}{\color[rgb]{0,0,0}$\tau^j\Lambda_m$}%
}}}}
\put(2431,-1051){\makebox(0,0)[lb]{\smash{{\SetFigFont{12}{14.4}{\rmdefault}{\mddefault}{\updefault}{\color[rgb]{0,0,0}$e_j$}%
}}}}
\put(4411,-2761){\makebox(0,0)[lb]{\smash{{\SetFigFont{12}{14.4}{\rmdefault}{\mddefault}{\updefault}{\color[rgb]{0,0,0}$\tau^j\partial\Lambda_n$}%
}}}}
\put(3736,-2761){\makebox(0,0)[lb]{\smash{{\SetFigFont{12}{14.4}{\rmdefault}{\mddefault}{\updefault}{\color[rgb]{0,0,0}$\tau^i\partial\Lambda_n$}%
}}}}
\end{picture}%

%% file: halfface.pspdftex
\begin{picture}(0,0)%
\includegraphics{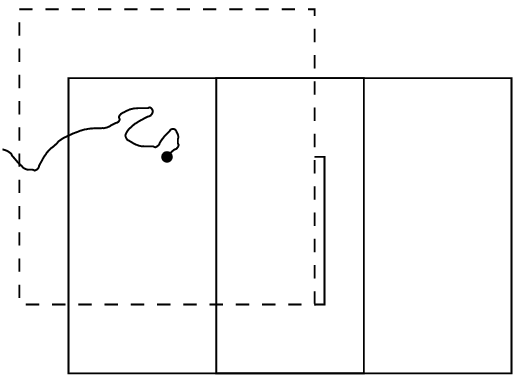}%
\end{picture}%
\setlength{\unitlength}{4144sp}%
\begingroup\makeatletter\ifx\SetFigFont\undefined%
\gdef\SetFigFont#1#2#3#4#5{%
  \reset@font\fontsize{#1}{#2pt}%
  \fontfamily{#3}\fontseries{#4}\fontshape{#5}%
  \selectfont}%
\fi\endgroup%
\begin{picture}(2351,1926)(137,-1210)
\put(811,-151){\makebox(0,0)[lb]{\smash{{\SetFigFont{12}{14.4}{\rmdefault}{\mddefault}{\updefault}{\color[rgb]{0,0,0}$y$}%
}}}}
\put(271,524){\makebox(0,0)[lb]{\smash{{\SetFigFont{12}{14.4}{\rmdefault}{\mddefault}{\updefault}{\color[rgb]{0,0,0}$y+\Lambda_N$}%
}}}}
\put(1621,-376){\makebox(0,0)[lb]{\smash{{\SetFigFont{12}{14.4}{\rmdefault}{\mddefault}{\updefault}{\color[rgb]{0,0,0}$F$}%
}}}}
\put(451,-1141){\makebox(0,0)[lb]{\smash{{\SetFigFont{12}{14.4}{\rmdefault}{\mddefault}{\updefault}{\color[rgb]{0,0,0}$\Lambda_{x'}$}%
}}}}
\put(2251,-1141){\makebox(0,0)[lb]{\smash{{\SetFigFont{12}{14.4}{\rmdefault}{\mddefault}{\updefault}{\color[rgb]{0,0,0}$\Lambda_x$}%
}}}}
\end{picture}%